\documentclass[12pt]{amsart}
\usepackage{amsmath,amsfonts,amsthm,amssymb,bbm,tikz,url}
\usetikzlibrary{decorations.pathmorphing}
\usetikzlibrary{calc}
\usepackage[margin=1.in]{geometry}

\newcommand{\reals}{\mathbb{R}}
\newcommand{\comps}{\mathbb{C}}

\newcommand{\nats}{\mathbb{N}}

\newtheorem{thm}{Theorem}[section]

\newtheorem{prop}[thm]{Proposition}
\newtheorem{cor}[thm]{Corollary}

\theoremstyle{definition}

\newtheorem{defn}[thm]{Definition}

\begin{document}
\title{Link Patterns and the Catalan Tree}

\author[S. Ng]{Stephen Ng}
\address{Department of Mathematics\\
University of Rochester\\
Rochester, NY 14627, USA}
\email{ng@math.rochester.edu}

%\date{Version: \today}
\maketitle

\begin{abstract}
	We demonstrate that a natural construction based on the two notions of insertion of a strand and finding the preimages of Temperley-Lieb algebra generators give an inductive means to generate all link patterns of a given number of strands. It is shown that the structure of the Catalan tree (as defined by Julian West) arises in the process of this induction, and that it can be exploited to give some refined enumerations of link patterns.
\end{abstract}

\section{Introduction}
Link patterns, though most recently associated to fully packed loops through the Razumov-Stroganov correspondence proved by Cantini and Sportiello \cite{CantiniSportiello}, have a long history in the development of the graphical calculus for the representation theory of the Lie algebra $\mathfrak{sl}_2$ and, more generally, the closely related $q$-deformation, $U_q(\mathfrak{sl}_2)$ (See \cite{Chari1994,Frenkel1997,Kauffman1994,KuperbergSpiders,Penrose}). In this framework, link patterns are a special subclass of one-dimensional representations of $U_q(\mathfrak{sl}_2)$. Following the general setup of Schur-Weyl duality, it is known that the Temperley-Lieb algebra is the $U_q(\mathfrak{sl}_2)$ invariant algebra of operators acting on link patterns. We omit these details in favor of a simpler approach. We follow an exposition in the same spirit as the completely graphical description given by Cantini and Sportiello and the many others who have worked on problems related to the Razumov-Stroganov correspondence \cite{CantiniSportiello,deGier,Zuber}. Even so, we point out that link patterns are natural objects which arise in many different contexts and that link patterns are special objects in the Catalan combinatorial space precisely because of the very nice action of the Temperley-Lieb algebra.

After introducing link patterns and the Temperley-Lieb algebra, this paper seeks to resolve the issue of efficiently generating the set of all link patterns of a given number of strands. For computer simulations involving link patterns, efficiency is important because the number of link patterns of $n$ strands grows asymptotically as $O(\frac{4^n}{n^{3/2}})$. The method is an inductive one commonly known as the ECO method of constructing a generating tree (see \cite{bdpp,West}): given all link patterns of $n-1$ strands, we describe a means to generate all link patterns of $n$ strands by ``inverting'' the action of a fixed Temperley-Lieb algebra generator. The generating tree resulting from this construction, also known as the Catalan tree, induces bijections to other well known objects in the Catalan combinatorial space, and allows us to obtain refined enumerations of link patterns.

The structure of the remainder of the paper follows the pattern set in the previous paragraph. Section \ref{sec:catalanTree} introduces the Catalan tree, following West's exposition for pattern avoiding permutations \cite{West}. Section \ref{sec:lp} introduces link patterns, the Temperley-Lieb algebra, and establishes the Catalan tree structure for link patterns. Lastly, section \ref{sec:enumerations} exploits the bijection to Dyck paths induced by the Catalan tree structure to obtain refined enumerations of link patterns. 
\section{The Catalan Tree}
\label{sec:catalanTree}
Many combinatorial objects are known to be counted by the Catalan numbers. The Catalan numbers are typically introduced by way of the inductive formula, $C_n= \sum_{i=0}^n C_i C_{n-i}$ with $C_0=1$, and in like manner, the Catalan Tree gives an inductive means to generate objects counted by $C_n=\frac{1}{n+1}\binom{2n}{n}$.
We define the Catalan Tree (following West) via the following labels and the succession rule:
\begin{enumerate}
	\item The root of the Catalan tree has label 2.
	\item A node with label $i$ will have a total of $i$ children with the distinct labels $2, 3, \ldots, i+1$.
\end{enumerate}
It is customary to write 
\begin{equation*}
	(k) \to (2)(3) \cdots (k+1)
\end{equation*}
to indicate that nodes with label $k$ give rise to $k$ children with labels $2,3, \ldots, k+1$. Related generating trees constructed by alternative succession rules have come to be known collectively as the ECO method, after the technique described abstractly by Barcucci, Del Lungo, Pergola, and Pinzani.
See Figure \ref{fig:CatalanTree} to see the Catalan tree up to level 2.

\begin{figure}[h]
	\begin{center}
\begin{tikzpicture}
	[every node/.style={draw=black, circle}, level 1/.style={sibling distance=6em}, level 2/.style={sibling distance=2em}, edge from parent/.style={draw=black!50}]
	\node {2}
		child {node {2}
			child {node {2}}
			child {node {3}}
		}
		child {node {3}
			child {node {2}}
			child {node {3}}
			child {node {4}}
		};
\end{tikzpicture}
	\end{center}
	\caption{The Catalan tree displayed up to level 2.}
	\label{fig:CatalanTree}
\end{figure}
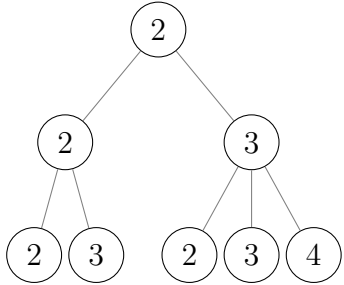
West introduced this tree to describe an inductive method to generate all permutations of length $n$ which avoid a fixed pattern, $\tau$, of length three. We shall only concern ourselves with the case $\tau=123$, and we follow West's terminology and exposition.
\begin{defn}
	Let $\sigma \in S_n$ and let $\tau=123 \in S_3$. We say that $\sigma$ is \emph{$123$-avoiding} if there is no $1 \leq i_{1} < i_{2} < i_{3} \leq n$ such that $\sigma(i_1) < \sigma(i_2) < \sigma(i_3)$. 
\end{defn}

Let $S_n(123)$ denote all permutations of size $n$ which are $123$-avoiding. Let us construct an element of $S_{n+1}(123)$ by the following \emph{insertion} process: given a permutation $\sigma \in S_n(123)$, obtain a new permutation $(\sigma(1),\sigma(2),\ldots,\sigma(i-1),(n+1), \sigma(i),\ldots,\sigma(n))$, where the entry $n+1$ is in position $i$ and the new permutation is $123$-avoiding. For an arbitrary $\sigma \in S_n(123)$, let $i$ be the first index such that $\sigma(i) < \sigma(i+1)$. Then one can see that the insertion of $n+1$ into any of the positions $1, 2, 3, \ldots, i, i+1$ yields an element of $S_{n+1}(123)$. Thus, by letting the root of the tree be the trivial permutation in $S_1$ and allowing the succession rule to be all permutations obtained by insertion of $n+1$ which preserve $123$-avoidance, we see that we recover the structure of the Catalan tree. See Figure \ref{fig:123avoid} for a diagram of the tree of these permutations up to $S_3(123)$.

\begin{figure}[h]
	\begin{center}
		\begin{tikzpicture}
			[every node/.style={draw=black, rectangle}, level 1/.style={sibling distance=7em}, level 2/.style={sibling distance=2.5em},edge from parent/.style={draw=black!50}]
	\node {1}
		child {node {12}
			child {node {132}}
			child {node {312}}
		}
		child {node {21}
			child {node {231}}
			child {node {213}}
			child {node {321}}
		};
		\end{tikzpicture}
	\end{center}
	\caption{The Catalan tree of $123$-avoiding permutations up to level 2.}
	\label{fig:123avoid}
\end{figure}
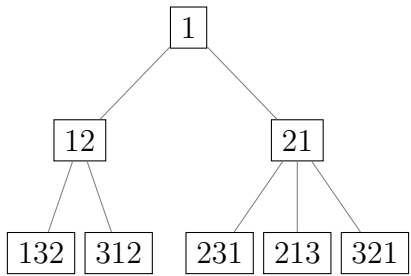

It is well known that $\tau$-avoiding permutations for $\tau \in S_3$ are counted by the Catalan numbers (see \cite{knuth}).

\section{Link Patterns and the Catalan Tree}
\label{sec:lp}
We now introduce link patterns and demonstrate how they may be generated inductively using the Catalan Tree. 
\begin{defn}
	Fix $n$, a positive integer. Let $D$ be the closed unit disk in $\comps$. Label the $2n$th roots of unity $0,1,2, \ldots, 2n-1$ such that $k$ corresponds to $e^{2\pi i (k/2n)}$. A \emph{matching} is a partition of the set of points $\{0, 1, 2, \ldots, 2n-1\}$ into pairs of points. We represent this graphically by drawing a line from $i$ to $j$ if $i$ and $j$ form a pair in the given partition. A \emph{link pattern of $n$ strands} is a matching of $\{0,1,2, \ldots, 2n-1\}$ in which no lines cross in the graphical representation. Figures \ref{fig:lp4example} and \ref{fig:4strandlps} provide examples.
\end{defn}
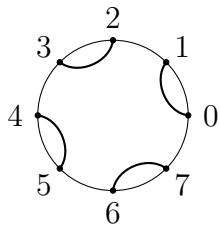
\begin{figure}[h]
	\begin{center}
		\begin{tikzpicture}
			\draw (0,0) circle (1cm);
			\foreach \x in {0, 1, ...,7}
			{
				\draw[fill] (45*\x:1) circle (1pt);
				\node at (45*\x:1.3) {$\x$};
			}
			\foreach \y in {0,90, ...,270}
			\draw[thick, bend left=60] (\y:1) to (\y+45:1);
		\end{tikzpicture}
	\end{center}
	\caption{An example link pattern of $4$ strands. We omit the numerals in subsequent diagrams.}
	\label{fig:lp4example}
\end{figure}
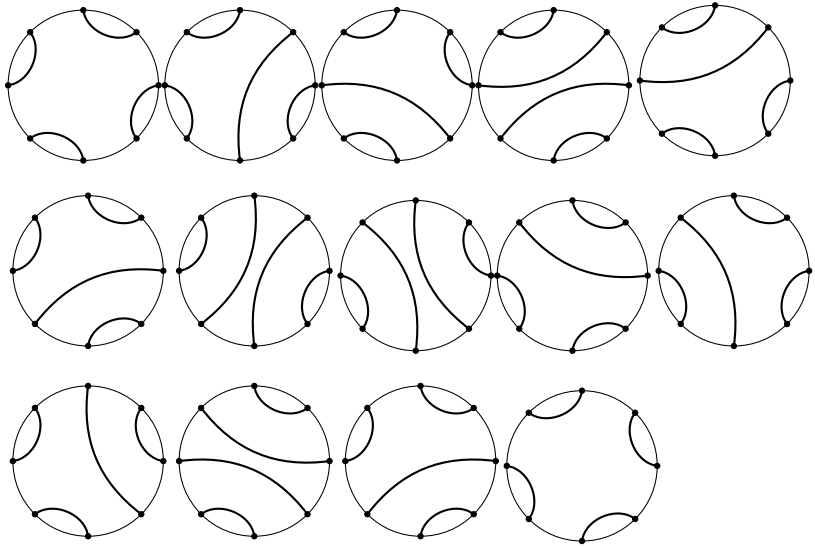
\begin{figure}[h]
	\begin{align*}
		%1
		&
		\begin{tikzpicture}
			\draw (0,0) circle (1cm);
			\foreach \x in {1,2,3,4,5,6,7,0}
			{
				\draw[fill] (45*\x:1) circle (1pt);
			}
			\foreach \y in {45,135, ...,315}
			\draw[thick, bend left=60] (\y:1) to (\y+45:1);
		\end{tikzpicture}
		%2
		\begin{tikzpicture}
			\draw (0,0) circle (1cm);
			\foreach \x in {0, 1, ...,7}
			{
				\draw[fill] (45*\x:1) circle (1pt);
			}
			\foreach \y in {315,90,180}
			\draw[thick, bend left=60] (\y:1) to (\y+45:1);
			\draw[thick,bend right=30] (45:1) to (270:1);
		\end{tikzpicture}
		%3
		\begin{tikzpicture}[rotate=-90]
			\draw (0,0) circle (1cm);
			\foreach \x in {0, 1, ...,7}
			{
				\draw[fill] (45*\x:1) circle (1pt);
			}
			\foreach \y in {315,90,180}
			\draw[thick, bend left=60] (\y:1) to (\y+45:1);
			\draw[thick,bend right=30] (45:1) to (270:1);
		\end{tikzpicture}
		%4
		\begin{tikzpicture}
			\draw (0,0) circle (1cm);
			\foreach \x in {0, 1, ...,7}
			{
				\draw[fill] (45*\x:1) circle (1pt);
			}
			\foreach \y in {90,270}
			\draw[thick, bend left=60] (\y:1) to (\y+45:1);
			\draw[thick, bend left=30] (45:1) to (180:1);
			\draw[thick, bend left=30] (225:1) to (0:1);
		\end{tikzpicture}
		%5
		\begin{tikzpicture}[rotate=135]
			\draw (0,0) circle (1cm);
			\foreach \x in {0, 1, ...,7}
			{
				\draw[fill] (45*\x:1) circle (1pt);
			}
			\foreach \y in {315,90,180}
			\draw[thick, bend left=60] (\y:1) to (\y+45:1);
			\draw[thick,bend right=30] (45:1) to (270:1);
		\end{tikzpicture} \\
		&
		%6
		\begin{tikzpicture}[rotate=-45]
			\draw (0,0) circle (1cm);
			\foreach \x in {0, 1, ...,7}
			{
				\draw[fill] (45*\x:1) circle (1pt);
			}
			\foreach \y in {315,90,180}
			\draw[thick, bend left=60] (\y:1) to (\y+45:1);
			\draw[thick,bend right=30] (45:1) to (270:1);
		\end{tikzpicture}
		%7
		\begin{tikzpicture}[rotate=45]
			\draw (0,0) circle (1cm);
			\foreach \x in {0, 1, ...,7}
			{
				\draw[fill] (45*\x:1) circle (1pt);
			}
			\foreach \y in {90,270}
			\draw[thick, bend left=60] (\y:1) to (\y+45:1);
			\draw[thick, bend left=30] (45:1) to (180:1);
			\draw[thick, bend left=30] (225:1) to (0:1);
		\end{tikzpicture}
		%8
		\begin{tikzpicture}[rotate=90]
			\draw (0,0) circle (1cm);
			\foreach \x in {0, 1, ...,7}
			{
				\draw[fill] (45*\x:1) circle (1pt);
			}
			\foreach \y in {90,270}
			\draw[thick, bend left=60] (\y:1) to (\y+45:1);
			\draw[thick, bend left=30] (45:1) to (180:1);
			\draw[thick, bend left=30] (225:1) to (0:1);
		\end{tikzpicture}
		%9
		\begin{tikzpicture}[rotate=90]
			\draw (0,0) circle (1cm);
			\foreach \x in {0, 1, ...,7}
			{
				\draw[fill] (45*\x:1) circle (1pt);
			}
			\foreach \y in {315,90,180}
			\draw[thick, bend left=60] (\y:1) to (\y+45:1);
			\draw[thick,bend right=30] (45:1) to (270:1);
		\end{tikzpicture}
		%10
		\begin{tikzpicture}[rotate=225]
			\draw (0,0) circle (1cm);
			\foreach \x in {0, 1, ...,7}
			{
				\draw[fill] (45*\x:1) circle (1pt);
			}
			\foreach \y in {315,90,180}
			\draw[thick, bend left=60] (\y:1) to (\y+45:1);
			\draw[thick,bend right=30] (45:1) to (270:1);
		\end{tikzpicture}\\
		&
		%11
		\begin{tikzpicture}[rotate=45]
			\draw (0,0) circle (1cm);
			\foreach \x in {0, 1, ...,7}
			{
				\draw[fill] (45*\x:1) circle (1pt);
			}
			\foreach \y in {315,90,180}
			\draw[thick, bend left=60] (\y:1) to (\y+45:1);
			\draw[thick,bend right=30] (45:1) to (270:1);
		\end{tikzpicture}
		%12
		\begin{tikzpicture}[rotate=135]
			\draw (0,0) circle (1cm);
			\foreach \x in {0, 1, ...,7}
			{
				\draw[fill] (45*\x:1) circle (1pt);
			}
			\foreach \y in {90,270}
			\draw[thick, bend left=60] (\y:1) to (\y+45:1);
			\draw[thick, bend left=30] (45:1) to (180:1);
			\draw[thick, bend left=30] (225:1) to (0:1);
		\end{tikzpicture}
		%13
		\begin{tikzpicture}[rotate=-45]
			\draw (0,0) circle (1cm);
			\foreach \x in {0, 1, ...,7}
			{
				\draw[fill] (45*\x:1) circle (1pt);
			}
			\foreach \y in {315,90,180}
			\draw[thick, bend left=60] (\y:1) to (\y+45:1);
			\draw[thick,bend right=30] (45:1) to (270:1);
		\end{tikzpicture}
		%14
		\begin{tikzpicture}
			\draw (0,0) circle (1cm);
			\foreach \x in {0, 1, ...,7}
			{
				\draw[fill] (45*\x:1) circle (1pt);
			}
			\foreach \y in {0,90, ...,270}
			\draw[thick, bend left=60] (\y:1) to (\y+45:1);
		\end{tikzpicture}
	\end{align*}
	\caption{All link patterns of four strands.}
	\label{fig:4strandlps}
\end{figure}
\begin{defn}
	Fix $n$, a positive integer. Let $LP_n$ denote the formal vector space over $\comps$ of all link patterns of $n$ strands and let the standard basis of link patterns be called the Temperley-Lieb basis of $LP_n$. Define the \emph{Temperley-Lieb algebra}, denoted $TL_n$, to be the subalgebra of the algebra of endomorphisms of $LP_n$ generated by elements $e_0, e_1, \ldots, e_{2n-1}$ subject to the relations:
\begin{enumerate}
	\item $e_i^2 = e_i$ for $i\in \{0,1,\ldots, 2n-1\}$,
	\item $e_{i} e_{i\pm 1} e_i = e_i$ for $i\in \{0,1,\ldots, 2n-1\}$,
	\item $e_i e_j = e_j e_i$ whenever $|i-j|>1$.
\end{enumerate}
The elements $e_0, e_1, \ldots, e_{2n-1}$ are known as Temperley-Lieb generators and with minor modifications, $TL_n$ plays a significant role in the representation theory of $U_q(\mathfrak{sl}_2)$ through a $q$-generalization of Schur-Weyl duality (See \cite{Chari1994,Frenkel1997,Jimbo1986}).
These relations admit a simple graphical realization seen in Figure \ref{fig:TLA}:
\begin{figure}[h]
	\begin{center}
		\begin{tikzpicture}[scale=.5]
			\draw (0,0) circle (1cm);
			\draw (0,0) circle (2cm);
			\foreach \x in {0,45,90,...,315}{
				\filldraw (\x:1cm) circle (2pt);
				\filldraw (\x:2cm) circle (2pt);}
			\draw[thick, bend right=90] (135:1cm) to (180:1cm);
			\draw[thick, bend left=90] (135:2cm) to (180:2cm);
			\foreach \y in {0,45,90,225,270,315}
				\draw[thick] (\y:1cm) to (\y:2cm);
		\end{tikzpicture}
		\begin{tikzpicture}[scale=.5,rotate=45]
			\draw (0,0) circle (1cm);
			\draw (0,0) circle (2cm);
			\foreach \x in {0,45,90,...,315}{
				\filldraw (\x:1cm) circle (2pt);
				\filldraw (\x:2cm) circle (2pt);}
			\draw[thick, bend right=90] (135:1cm) to (180:1cm);
			\draw[thick, bend left=90] (135:2cm) to (180:2cm);
			\foreach \y in {0,45,90,225,270,315}
				\draw[thick] (\y:1cm) to (\y:2cm);
		\end{tikzpicture}
		\begin{tikzpicture}[scale=.5,rotate=90]
			\draw (0,0) circle (1cm);
			\draw (0,0) circle (2cm);
			\foreach \x in {0,45,90,...,315}{
				\filldraw (\x:1cm) circle (2pt);
				\filldraw (\x:2cm) circle (2pt);}
			\draw[thick, bend right=90] (135:1cm) to (180:1cm);
			\draw[thick, bend left=90] (135:2cm) to (180:2cm);
			\foreach \y in {0,45,90,225,270,315}
				\draw[thick] (\y:1cm) to (\y:2cm);
		\end{tikzpicture}
		\ldots
	\end{center}
	\caption{The Temperley Lieb Algebra generators}
	\label{fig:TLA}
\end{figure}
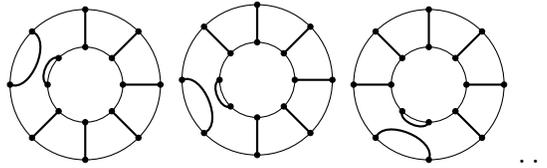
\end{defn}
In the graphical representation, if we label $2n$ points on the internal circle and external circle of the annulus in the same manner as for link patterns, then $e_i$ is the diagram in which all internal points are connected to their corresponding external points with a link except for $i$ and $i+1$. On the external circle the points labeled $i$ and $i+1$ are connected by a link, and likewise, a link connects $i$ and $i+1$ on the internal circle, where addition by $1$ is conducted modulo $2n$. The sum of elements in the Temperley-Lieb algebra is given by the formal sum of diagrams. The product of generators in the Temperley-Lieb algebra is given by concatenation of diagrams, and the first two relations are realized by ignoring loops and isotopies of diagrams:
\begin{equation*}
	e_i^2 = 
	\raisebox{-1.5cm}{
	\begin{tikzpicture}[scale=.5]
			\draw (0,0) circle (1cm);
			\draw (0,0) circle (2cm);
			\foreach \x in {0,45,90,...,315}{
				\filldraw (\x:1cm) circle (2pt);
				\filldraw (\x:2cm) circle (2pt);}
			\draw[thick, bend right=90] (135:1cm) to (180:1cm);
			\draw[thick, bend left=90] (135:2cm) to (180:2cm);
			\foreach \y in {0,45,90,225,270,315}
				\draw[thick] (\y:1cm) to (\y:2cm);
			\draw (0,0) circle (3cm);
			\foreach \x in {0,45,90,...,315}{
				\filldraw (\x:3cm) circle (3pt);}
			\draw[thick, bend right=90] (135:2cm) to (180:2cm);
			\draw[thick, bend left=30] (135:3cm) to (180:3cm);
			\foreach \y in {0,45,90,225,270,315}
				\draw[thick] (\y:1cm) to (\y:3cm);
			\end{tikzpicture}}
			= 
			\raisebox{-1.5cm}{
		\begin{tikzpicture}[scale=.7]
			\draw (0,0) circle (1cm);
			\draw (0,0) circle (2cm);
			\foreach \x in {0,45,90,...,315}{
				\filldraw (\x:1cm) circle (2pt);
				\filldraw (\x:2cm) circle (2pt);}
			\draw[thick, bend right=90] (135:1cm) to (180:1cm);
			\draw[thick, bend left=90] (135:2cm) to (180:2cm);
			\foreach \y in {0,45,90,225,270,315}
				\draw[thick] (\y:1cm) to (\y:2cm);
			\end{tikzpicture}} = e_i 
		\end{equation*}
		\begin{equation*}
			e_i e_{i+1} e_i = 
			\raisebox{-1.5cm}{
			\begin{tikzpicture}[scale=.35]
			\draw (0,0) circle (1cm);
			\draw (0,0) circle (2cm);
			\foreach \x in {0,45,90,...,315}{
				\filldraw (\x:1cm) circle (2pt);
				\filldraw (\x:2cm) circle (2pt);}
			\draw[thick, bend right=90] (135:1cm) to (180:1cm);
			\draw[thick, bend left=90] (135:2cm) to (180:2cm);
			\foreach \y in {0,45,90,225,270,315}
				\draw[thick] (\y:1cm) to (\y:2cm);
			\draw (0,0) circle (3cm);
			\draw (0,0) circle (4.5cm);
			\foreach \x in {0,45,90,...,315}{
				\filldraw (\x:3cm) circle (3pt);}
			\foreach \x in {0,45,90,...,315}{
				\filldraw (\x:4.5cm) circle (4pt);}
			\draw[thick, bend right=90] (180:2cm) to (225:2cm);
			\draw[thick, bend left=30] (180:3cm) to (225:3cm);
			\draw[thick, bend right=90] (135:3cm) to (180:3cm);
			\draw[thick, bend left=30] (135:4.5cm) to (180:4.5cm);
			\foreach \y in {0,45,90,225,270,315}
				\draw[thick] (\y:1cm) to (\y:2cm);
			\foreach \y in {0,45,90,135,270,315}
				\draw[thick] (\y:2cm) to (\y:3cm);
			\foreach \y in {0,45,90,225,270,315}
				\draw[thick] (\y:3cm) to (\y:4.5cm);
			\end{tikzpicture}} = 
			\raisebox{-1.5cm}{
		\begin{tikzpicture}[scale=.7]
			\draw (0,0) circle (1cm);
			\draw (0,0) circle (2cm);
			\foreach \x in {0,45,90,...,315}{
				\filldraw (\x:1cm) circle (2pt);
				\filldraw (\x:2cm) circle (2pt);}
			\draw[thick, bend right=90] (135:1cm) to (180:1cm);
			\draw[thick, bend left=90] (135:2cm) to (180:2cm);
			\foreach \y in {0,45,90,225,270,315}
				\draw[thick] (\y:1cm) to (\y:2cm);
			\end{tikzpicture}} = e_i.
\end{equation*}
The relations allow us to think of the Temperley-Lieb algebra, $TL_n$, intuitively as the set of diagrams on an annulus connecting $2n$ external points to $2n$ internal points in such a way that no two connecting lines cross. Internal loops may be ignored due to the first relation, and diagrams are equivalent up to isotopy due to the second relation.
The action of $TL_n$ on $LP_n$ is also given by concatenation: 
\begin{equation*}
		e_i(\raisebox{-.4cm}{
        \begin{tikzpicture}[scale=.5]
            \draw (0,0) circle (1cm);
            \foreach \y in {90,225,315}
                \draw[thick,bend left=90] (\y:1cm) to (\y+45:1cm);
                \draw[thick, bend left=45] (45:1cm) to (180:1cm);
            \foreach \x in {0,45,90,...,315}
                \filldraw (\x:1cm) circle (2pt);
	\end{tikzpicture}}\,)
		=
		       \raisebox{-1cm}{
        \begin{tikzpicture}[scale=.5]
            \draw (0,0) circle (1cm);
            \draw (0,0) circle (2cm);
            \foreach \x in {0,45,90,...,315}
                \filldraw (\x:1cm) circle (2pt);
            \foreach \x in {0,45,90,...,315}
                \filldraw (\x:2cm) circle (2pt);
            \foreach \y in {90,225,315}
                \draw[thick,bend left=90] (\y:1cm) to (\y+45:1cm);
                \draw[thick, bend left=45] (45:1cm) to (180:1cm);
            \foreach \x in {0,45,90,...,315}{
                \filldraw (\x:1cm) circle (2pt);
                \filldraw (\x:2cm) circle (2pt);}
            \draw[thick, bend right=90] (135:1cm) to (180:1cm);
            \draw[thick, bend left=90] (135:2cm) to (180:2cm);
            \foreach \y in {0,45,90,225,270,315}
                \draw[thick] (\y:1cm) to (\y:2cm);
        \end{tikzpicture}   
        } =
        \raisebox{-.7cm}{
        \begin{tikzpicture}[scale=.7]
            \draw (0,0) circle (1cm);
            \foreach \x in {0,45,90,...,315}
                \filldraw (\x:1cm) circle (2pt);
            \foreach \y in {45,135,225,315}
                \draw[thick,bend left=90] (\y:1cm) to (\y+45:1cm);
        \end{tikzpicture}   
        },
\end{equation*}
and the resulting diagrams are also taken to be invariant under isotopy and removal of loops.
\begin{defn}
	The \emph{insertion of a strand at $i$ and $i+1$} is a function from the Temperley-Lieb basis of $LP_n$ to the Temperley-Lieb basis of $LP_{n+1}$ defined by taking $\pi \in LP_n$, inserting two additional points on the unit circle in between $i-1$ and $i$, relabeling the points sequentially so that the newly added points are labeled $i$ and $i+1$, inserting a link between $i$ and $i+1$, and repositioning the points to lie on the roots of unity. An example of the insertion of a strand procedure is seen in Figure \ref{fig:insertexample}.
	
	Let $A_i$ be the subset of $LP_n$ which is obtained from $LP_{n-1}$ by the insertion of a strand at $i$ and $i+1$. Define the \emph{deletion of a strand at $i$ and $i+1$} as a function from $A_i$ to $LP_{n-1}$ in which the strand connecting $i$ and $i+1$ and the points $i$ and $i+1$ are removed, and the points $i+2, i+3, \ldots, 2n-1$ are relabeled and repositioned to fill in the gap caused by the deletion of $i$ and $i+1$ in the sequence $0, 1, \ldots, 2n-3$. 
\end{defn}
\begin{figure}[h]
	\begin{equation*}
		\raisebox{-1.5cm}{
		\begin{tikzpicture}
			\draw (0,0) circle (1cm);
			\foreach \x in {0, 1, ...,7}
			{
				\draw[fill] (45*\x:1) circle (1pt);
				\node at (45*\x:1.3) {$\x$};
			}
			\foreach \y in {0,90, ...,270}
			\draw[thick, bend left=60] (\y:1) to (\y+45:1);
		\end{tikzpicture}
	} \to 
	\raisebox{-1.5cm}{
		\begin{tikzpicture}
			\draw (0,0) circle (1cm);
			\foreach \y in {0,72,288}
			\draw[thick, bend left=60] (\y:1) to (\y+36:1);
			\draw[thick,bend left=60] (144:1) to (252:1);
			\draw[thick,bend left=60,color=red] (144+36:1) to (144+72:1);
			\foreach \x in {0, 1, ...,9}
			{
				\draw[fill] (36*\x:1) circle (1pt);
				\node at (36*\x:1.3) {$\x$};
			}
		\end{tikzpicture}
	}
	\end{equation*}
	\caption{An example of an insertion of a link between $5$ and $6$. The inserted strand has been highlighted in red.}
	\label{fig:insertexample}
\end{figure}
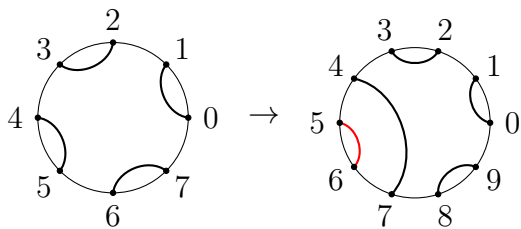
We now give an alternate representation of a link pattern in the upper half-plane of $\reals^2$. 
\begin{defn}
	By \emph{cut and unfold between $i$ and $i+1$}, we mean the operation on a fixed link pattern in which we cut the unit circle between $i$ and $i+1$ and unwrap the circle onto a horizontal line, taking care to preserve links between numbers in the half-plane above the line, and lastly, embedding the resulting diagram into the upper half-plane. In this cut and unfolded form, links subdivide the half-plane into components, and we define \emph{outermost links} to be those links which are incident to the unbounded component. Given a fixed link pattern of n strands, $\pi$, and a fixed choice of a cut and unfold, let the \emph{exposure number of $\pi$} be the number of outermost links. An example of a cut and unfold and its associated outermost links is given in Figure \ref{fig:cutUnfold}.
\end{defn}
\begin{figure}[h]
	\begin{equation*}
		\raisebox{-1.5cm}{
		\begin{tikzpicture}
			\draw (0,0) circle (1cm);
			\foreach \x in {0, 1, ...,7}
			{
				\draw[fill] (45*\x:1) circle (1pt);
				\node at (45*\x:1.3) {$\x$};
			}
			\foreach \y in {90,180,270}
			\draw[thick, bend left=60] (\y:1) to (\y+45:1);
			\draw[dashed] (22.5:0.8) to (22.5:1.25);
			\draw[color=red,thick, bend left=60] (0:1) to (45:1);
		\end{tikzpicture}
	} \to 
	\raisebox{-1cm}{
		\begin{tikzpicture}[scale=.5]
		\draw (-.5,0) -- (7.5,0);
		\draw[thick] (1,0) arc (180:0:.5);
		\draw[color=red,thick,bend left=90] (0,0) to (7,0);
		\draw[thick] (3,0) arc (180:0:.5);
		\draw[thick] (5,0) arc (180:0:.5);
			\foreach \x in {1,2, ...,7}
			{
				\draw[fill] (\x,0) circle (2pt);
				\node at (\x-1,-.5) {$\footnotesize{\x}$};
			}
			\node at (7,-.5) {$\footnotesize{0}$};
		\end{tikzpicture}
	}
	\end{equation*}
	\caption{Cut and unfold between 0 and 1. Outermost links have been highlighted in red.}
	\label{fig:cutUnfold}
\end{figure}
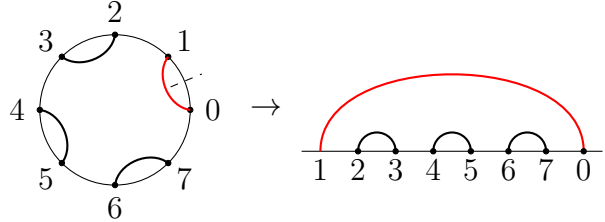
Outermost links with respect to a fixed cut and unfold allow us to characterize the preimages of the Temperley-Lieb generators $e_i$. 
\begin{prop}

	Without loss of generality, in all that follows consider $e_{2n-1}$ and note that results for other $e_i$ can be reduced to this special case by repeated rotations. Let $\pi$ be a link pattern in $LP_{n-1}$ suppose that $\pi$ has exposure number $k$ when cut and unfolded between $2n-2$ and $0$. Let $\pi^\prime \in LP_{n}$ be the result of inserting a link at $2n-1$ and $0$ to $\pi$. Then the preimage of $\pi^\prime$ under the map $e_{2n-1}$ is characterized by $k+1$ distinct link patterns in $LP_n$.
	\label{prop:preimage}
\end{prop}
\begin{proof}
	Given $\pi \in LP_{n-1}$ with exposure number $k$ when cut and unfolded between $2n-2$ and $0$ and letting $\pi^\prime \in LP_n$ be the result of insertion of a strand at $2n-1$ and $0$, note that $\pi^\prime$ has exposure number $k+1$ when cut and unfolded between $0$ and $1$. We seek to define operations $M_i(\cdot)$ such that $e_{2n-1}(M_i(\pi^\prime)) = \pi^\prime$ which are distinct for $i \in \{0, 1, 2, \ldots, k\}$. Let $M_0 (\pi^\prime) = \pi^\prime$. For the remaining $i \in \{1, 2, \ldots, k \}$, label the outermost links $\{1, 2, \ldots, k \}$ from left to right and define $M_i(\pi^\prime)$ as the result of the following process with the $i$th outermost link:
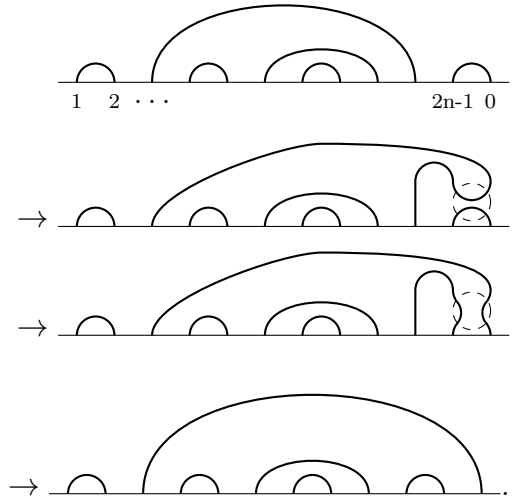
\begin{align*}
	\begin{tikzpicture}[scale=.5]
		\draw (-.5,0) -- (11.5,0);
		\draw[thick] (0,0) arc (180:0:.5);
		\draw[thick,bend left=90] (2,0) to (9,0);
		\draw[thick] (3,0) arc (180:0:.5);
		\draw[thick,bend left=90] (5,0) to (8,0);
		\draw[thick] (6,0) arc (180:0:.5);
		\draw[thick] (10,0) arc (180:0:.5);
		\node at (0,-.5) {\tiny{1}};
		\node at (1,-.5) {\tiny{2}};
		\node at (2,-.5) {$\ldots$};
		\node at (10,-.5) {\tiny{2n-1}};
		\node at (11,-.5) {\tiny{0}};
	\end{tikzpicture} \\
	\to
	\begin{tikzpicture}[scale=.5]
		\draw (-.5,0) -- (11.5,0);
		\draw[thick] (0,0) arc (180:0:.5);
		\draw[thick] (2,0) .. controls +(0,1) and (5.5,2.2) .. (6.5,2.2) .. controls +(1,0) and (11,2.2) .. (11,1.2) arc (0:-180:.5) arc (0:180:.5) -- (9,0);
		\draw[dashed] (10.5,.65) circle (.5);
		\draw[thick] (3,0) arc (180:0:.5);
		\draw[thick,bend left=90] (5,0) to (8,0);
		\draw[thick] (6,0) arc (180:0:.5);
		\draw[thick] (10,0) arc (180:0:.5);
	\end{tikzpicture}\\
	\to
	\begin{tikzpicture}[scale=.5]
		\draw (-.5,0) -- (11.5,0);
		\draw[thick] (0,0) arc (180:0:.5);
		\draw[thick] (2,0) .. controls +(0,1) and (5.5,2.2) .. (6.5,2.2) .. controls +(1,0) and (11,2.2) .. (11,1.2) arc (0:-180:.5) arc (0:180:.5) -- (9,0);
		\draw[thick] (3,0) arc (180:0:.5);
		\draw[thick,bend left=90] (5,0) to (8,0);
		\draw[thick] (6,0) arc (180:0:.5);
		\draw[thick] (10,0) arc (180:0:.5);
		\draw[dashed,fill=white] (10.5,.65) circle (.5);
		\draw[thick,bend right=45] (10.925,.925) to (10.9,.3);
		\draw[thick,bend left=45] (10.075,.925) to (10.1,.3);
	\end{tikzpicture}\\
	\to
	\begin{tikzpicture}[scale=.5]
		\draw (-.5,0) -- (11.5,0);
		\draw[thick] (0,0) arc (180:0:.5);
		\draw[thick, bend left=90] (2,0) to (11,0);
		\draw[thick] (3,0) arc (180:0:.5);
		\draw[thick,bend left=90] (5,0) to (8,0);
		\draw[thick] (6,0) arc (180:0:.5);
		\draw[thick] (9,0) arc (180:0:.5);
	\end{tikzpicture}.
\end{align*}
In the above, we have illustrated $M_2(\cdot)$. Observe that in general, $M_i(\cdot)$ is dependent upon $\pi$, a given cut and unfold, and the insertion of a new strand. In words, this process is carried out by ``dragging'' the $i$th outermost link near to the newly added link between $2n-1$ and $0$ and implementing the change $\raisebox{-.1cm}{\begin{tikzpicture}[scale=.5] \draw[dashed] (0,0) circle (.5); \draw[thick, bend left=45] (45:.5) to (135:.5); \draw[thick, bend right=45] (-45:.5) to (-135:.5);\end{tikzpicture}} \to \raisebox{-.1cm}{\begin{tikzpicture}[scale=.5,rotate=90] \draw[dashed] (0,0) circle (.5); \draw[thick, bend left=45] (45:.5) to (135:.5); \draw[thick, bend right=45] (-45:.5) to (-135:.5);\end{tikzpicture} } $. It is straightforward to check that $e_{2n-1}(M_i(\pi^\prime))=\pi^\prime$. Thus we have $k+1$ distinct basis elements in the preimage of $e_{2n-1}$ given by $M_0(\pi^\prime), \ldots, M_{k}(\pi^\prime)$. To see that this exhausts all possibilities, observe that the operator $e_{2n-1}$ connects the links connected to $2n-1$ and $0$ and that this new link formed is an outermost link. 
\end{proof}
\begin{thm}
	The Temperley-Lieb basis of $LP_n$ is generated from the Temperley-Lieb basis of $LP_{n-1}$ inductively via a succession rule. Let $\pi \in LP_{n-1}$ have exposure number $k$ when cut and unfolded between $2n-2$ and $0$. If $\pi^\prime \in LP_n$ is the link pattern obtained from $\pi$ from an insertion of a link between $2n-1$ and $0$ and cut and unfolded between $0$ and $1$, then there are $k+1$ successors of $\pi$, and they are the preimages of $\pi^\prime$ under $e_{2n-1}$, $M_0(\pi^\prime), M_1(\pi^\prime), \ldots, M_k(\pi^\prime)$, and 
	\begin{equation*}
		LP_n = \bigoplus_{\pi \in LP_{n-1}} \mathrm{span}\{M_0(\pi^\prime),\ldots, M_{k_{\pi}}(\pi^\prime) \}.
	\end{equation*}
	\label{thm:successionRule}
\end{thm}
\begin{proof}
	We observe that the image of the operator $e_{2n-1}$ in $LP_n$ is characterized by the basis containing all link patterns with a link connecting sites $2n-1$ and $0$. Let us call this basis $\beta$. Each element of this basis is indexed by an element of the Temperley-Lieb basis of $LP_{n-1}$ obtained by the insertion of a link between $2n-1$ and $0$. Hence, for each link pattern $\pi$ in the Temperley-Lieb basis of $LP_{n-1}$ we have an associated element of $\beta$ denoted $\pi^\prime$. We then partition the Temperley-Lieb basis of $LP_n$ into the preimages of each element $\pi^\prime \in \beta$ under the map $e_{2n-1}$, and these preimages are spanned by $M_0(\pi^\prime), \ldots, M_k(\pi^\prime)$, as described in the proof of Proposition \ref{prop:preimage}.
\end{proof}

Theorem \ref{thm:successionRule} gives the succession rule which allows us to inductively construct the Temperley-Lieb bases using the structure of the Catalan Tree. Each link pattern $\pi$ is labeled by $k+1$, where $k$ is the number of outermost links when $\pi$ is cut and unfolded between $0$ and $1$. It is convenient to organize the successors of $\pi$ in increasing order of the total number of outermost links. This yields the local picture:
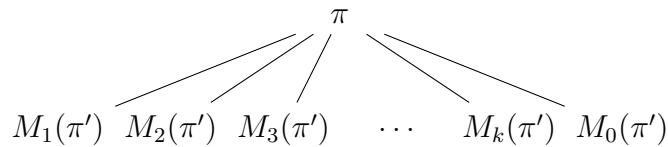
\begin{equation*}
	\begin{tikzpicture}
		[every node/.style={rectangle, minimum width=1.5cm}]
		\node(a) {$\pi$}
			child {node(b1) {$M_1(\pi^\prime)$}}
			child {node (b2) {$M_2(\pi^\prime)$}}
			child {node (b3) {$M_3(\pi^\prime)$}}
			child[white] {node[black] (b4) {$\ldots$}}
			child {node (bk) {$M_k(\pi^\prime)$}}
			child {node(b0) {$M_0(\pi^\prime)$}};
	\end{tikzpicture}
\end{equation*}
where $\pi$ is a Temperley-Lieb basis element of $LP_{n-1}$ with exposure number $k$ when cut and unfolded between $0$ and $1$, and $\pi^\prime$ is the result of inserting a link between $2n-1$ and $0$ to $\pi$, and $M_0, M_1, \ldots, M_k$ are the operations described in the proof of Proposition \ref{prop:preimage}. Observe that $M_k(\pi^\prime)$ has exposure number $k$ when $k\neq0$ and that $M_0(\pi^\prime)$ has exposure number $k+1$, so that this is consistent with the convention to order successors increasing in exposure number. In Figure \ref{fig:lpTree}, we provide an example of the Catalan Tree for link patterns up to level 2. 
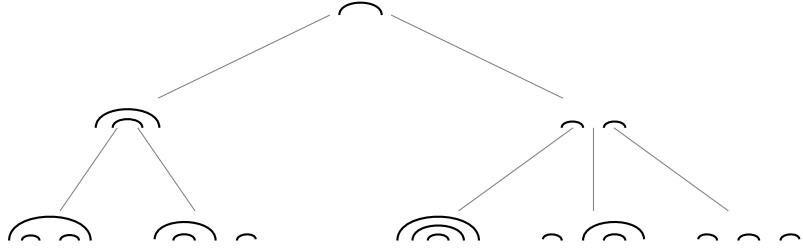
\begin{figure}[h]
	\begin{center}
		\begin{tikzpicture}
			[every node/.style={draw=white,line width=3pt, rectangle,minimum width=3em, minimum height=.7em}, level 1/.style={sibling distance=15em}, level 2/.style={sibling distance=5em}, edge from parent/.style={draw=black!50 }]
	\node(a) {}
		child {node (b1) {}
			child {node (c1) {}}
			child {node (c2) {}}
		}
		child {node (b2) {}
			child {node (c3) {}}
			child {node (c4) {}}
			child {node (c5) {}}
		};
		\draw[thick, bend left=90] (a.-145) to (a.-35);
		\draw[thick, bend left=90] (b1.-135) to (b1.-45);
		\draw[thick, bend left=90] (b1.-155) to (b1.-25);
		\draw[thick, bend left=90] (b2.-155) to (b2.-125);
		\draw[thick, bend left=90] (b2.-55) to (b2.-25);
		\draw[thick, bend left=90] (c1.-152.5) to (c1.-125);
		\draw[thick, bend left=90] (c1.-55) to (c1.-27.5);
		\draw[thick, bend left=90] (c1.-160) to (c1.-20);
		\draw[thick, bend left=90] (c2.-165) to (c2.-55);
		\draw[thick, bend left=90] (c2.-155) to (c2.-125);
		\draw[thick, bend left=90] (c2.-25) to (c2.-15);
		\draw[thick, bend left=90] (c3.-160) to (c3.-20);
		\draw[thick, bend left=90] (c3.-150) to (c3.-30);
		\draw[thick, bend left=90] (c3.-125) to (c3.-55);
		\draw[thick, bend left=90] (c4.-165) to (c4.-155);
		\draw[thick, bend left=90] (c4.-125) to (c4.-15);
		\draw[thick, bend left=90] (c4.-55) to (c4.-25);
		\draw[thick, bend left=90] (c5.-165) to (c5.-155);
		\draw[thick, bend left=90] (c5.-125) to (c5.-55);
		\draw[thick, bend left=90] (c5.-25) to (c5.-15);
		\end{tikzpicture}
	\end{center}
	\caption{The Catalan tree of link patterns up to level 2.}
	\label{fig:lpTree}
\end{figure}

\section{Refined enumerations of link patterns}
\label{sec:enumerations}
We now demonstrate that the Catalan tree construction for link patterns yields a useful bijection to yet another member of the class of Catalan objects, Dyck paths. Refined enumerations of Dyck paths are well known (see \cite{bdpp,Deutsch}), and they induce refined enumerations of link patterns.

\begin{defn}
	A \emph{Dyck path of semilength $n$} is a sequence of $2n+1$ points in $\nats \times \nats$ in which consecutive points differ by one of the vectors $(1,1)$ (an \emph{ascent}) or $(1,-1)$ (a \emph{descent}), the first point in the sequence is $(0,0)$ and the last point  in the sequence is $(2n,0)$. A \emph{peak} of a Dyck path is a point of the sequence preceded by an ascent and followed by a descent. Lastly, the \emph{last descent length} of a Dyck path is the maximal number of consecutive descents which precede the last point $(2n,0)$.
\end{defn}
\begin{prop}{\cite{bdpp,Deutsch}}
	The number of Dyck paths of semilength $n$ are counted by the Catalan numbers. The number of Dyck paths of semilength $n$ and last descent length $k$ is $\binom{2n-k}{n} \frac{k}{2n-k}$. The number of Dyck paths of semilength $n$ and $k$ peaks is $\frac{1}{n}\binom{n}{k} \binom{n}{k-1}$.
\end{prop}

Dyck paths may be inductively constructed with the structure of the Catalan tree by the operation of adding a new peak at each of the points along the last descent. For example, a Dyck path with last descent length 3 will have four descendents with last descent lengths 1,2,3, and 4, respectively:
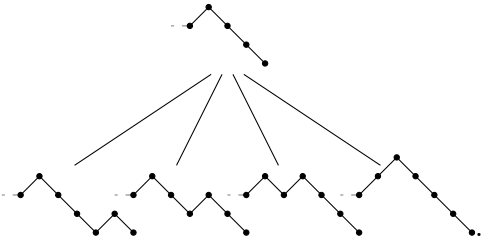
\begin{equation*}
	\begin{tikzpicture}
		[every node/.style={rectangle, minimum width=1.75cm}]
		\node(a) {}
		child {node(b1) {}}
		child {node (b2) {}}
		child {node (b3) {}}
		child {node(b0) {}};
		\draw (a) ++(-.5,.5) -- ++(.25,.25) -- ++(.75,-.75);
		\draw[dashed,gray] (a) ++(-.5,.5) -- ++(-.25,0);
		\draw[fill] (a) ++(-.5,.5) circle (1pt);
		\draw[fill] (a) ++(-.25,.75) circle (1pt);
		\draw[fill] (a) ++(0,.5) circle (1pt);
		\draw[fill] (a) ++(0.25,.25) circle (1pt);
		\draw[fill] (a) ++(0.5,0) circle (1pt);
		\draw (b1) ++(-.5,-.25) -- ++(.25,.25) -- ++(.75,-.75) -- ++(.25,.25) -- ++(.25,-.25);
		\draw[dashed,gray] (b1) ++(-.5,-.25) -- ++(-.25,0);
		\draw[fill] (b1) ++(-.5,-.25) circle (1pt);
		\draw[fill] (b1) ++(-.25,0) circle (1pt);
		\draw[fill] (b1) ++(0,-.25) circle (1pt);
		\draw[fill] (b1) ++(0.25,-.5) circle (1pt);
		\draw[fill] (b1) ++(0.5,-.75) circle (1pt);
		\draw[fill] (b1) ++(0.75,-.5) circle (1pt);
		\draw[fill] (b1) ++(1,-.75) circle (1pt);
		\draw (b2) ++(-.5,-.25) -- ++(.25,.25) -- ++(.5,-.5) -- ++(.25,.25) -- ++(.5,-.5);
		\draw[dashed,gray] (b2) ++(-.5,-.25) -- ++(-.25,0);
		\draw[fill] (b2) ++(-.5,-.25) circle (1pt);
		\draw[fill] (b2) ++(-.25,0) circle (1pt);
		\draw[fill] (b2) ++(0,-.25) circle (1pt);
		\draw[fill] (b2) ++(0.25,-.5) circle (1pt);
		\draw[fill] (b2) ++(0.5,-.25) circle (1pt);
		\draw[fill] (b2) ++(0.75,-.5) circle (1pt);
		\draw[fill] (b2) ++(1,-.75) circle (1pt);
		\draw (b3) ++(-.5,-.25) -- ++(.25,.25) -- ++(.25,-.25) -- ++(.25,.25) -- ++(.75,-.75);
		\draw[dashed,gray] (b3) ++(-.5,-.25) -- ++(-.25,0);
		\draw[fill] (b3) ++(-.5,-.25) circle (1pt);
		\draw[fill] (b3) ++(-.25,0) circle (1pt);
		\draw[fill] (b3) ++(0,-.25) circle (1pt);
		\draw[fill] (b3) ++(0.25,0) circle (1pt);
		\draw[fill] (b3) ++(0.5,-.25) circle (1pt);
		\draw[fill] (b3) ++(0.75,-.5) circle (1pt);
		\draw[fill] (b3) ++(1,-.75) circle (1pt);
		\draw (b0) ++(-.5,-.25) -- ++(.5,.5) -- ++(1,-1);
		\draw[dashed,gray] (b0) ++(-.5,-.25) -- ++(-.25,0);
		\draw[fill] (b0) ++(-.5,-.25) circle (1pt);
		\draw[fill] (b0) ++(-.25,0) circle (1pt);
		\draw[fill] (b0) ++(0,.25) circle (1pt);
		\draw[fill] (b0) ++(0.25,0) circle (1pt);
		\draw[fill] (b0) ++(0.5,-.25) circle (1pt);
		\draw[fill] (b0) ++(0.75,-.5) circle (1pt);
		\draw[fill] (b0) ++(1,-.75) circle (1pt);
	\end{tikzpicture}.
\end{equation*}
Thus, we see that the bijection between Dyck paths and link patterns induced by the Catalan tree maps the last descent length of a given Dyck path to the exposure number of a link pattern relative to a fixed cut and unfold.
\begin{defn}
	Fix a link pattern of $n$ strands and fix a particular cut and unfold. Construct the Catalan tree up to level $n$ and consider the path from the root to the given link pattern. Let the \emph{interaction number} of the link pattern relative to a fixed cut and unfold be the number of edges in this path corresponding to an operation $M_k(\cdot)$ with $k\neq0$. 
\end{defn}
One may observe that a move along an edge of the Catalan tree corresponding to $M_0(\cdot)$ does not increase the interaction number, and all other transitions increase the interaction number by 1. Likewise, in the Dyck path picture, the transition corresponding to a Dyck path with last descent length $k$ going to a Dyck path with last descent length $k+1$ does not increase the number of peaks, and all other transitions increase the number of peaks by 1. Accounting for the fact that the root of the Catalan tree for Dyck paths has one peak and that the root of the Catalan tree for link patterns has interaction number $0$, it follows that the number of peaks in a Dyck path corresponds to the interaction number $+1$ of a link pattern relative to a given cut and unfold.
\begin{cor}
	The number of link patterns of $n$ strands is given by the Catalan numbers. Fix a particular cut and unfold. The number of link patterns of $n$ strands with exposure number $k$ is $\binom{2n-k}{n} \frac{k}{2n-k}$. The number of link patterns of $n$ strands and interaction number $\ell$ is $\frac{1}{n} \binom{n}{\ell+1} \binom{n}{\ell}$.
\end{cor}
\bibliographystyle{amsplain}
\bibliography{lpCatalan}
\end{document}